\documentclass[11pt]{article}

\usepackage[UTF8,scheme=plain]{ctex}
\usepackage[a4paper,margin=1in]{geometry}
\usepackage{amsmath,amssymb,amsthm,mathtools}
\usepackage[T1]{fontenc}
\usepackage{lmodern}
\usepackage{microtype}
\usepackage{xcolor}
\usepackage{hyperref}
\hypersetup{
  colorlinks=true,
  linkcolor=blue!55!black,
  citecolor=blue!55!black,
  urlcolor=blue!55!black,
  pdftitle={Hierarchical Melnikov Realization: Logarithmic Factor Improvement to Hilbert Number Lower Bounds},
  pdfauthor={Chaoyang Qin, Xiaoming Sun},
  pdfsubject={Hilbert's sixteenth problem, limit cycles, and Melnikov functions}
}

\newtheorem{theorem}{Theorem}[section]
\newtheorem{lemma}[theorem]{Lemma}
\newtheorem{proposition}[theorem]{Proposition}
\newtheorem{corollary}[theorem]{Corollary}
\theoremstyle{remark}

\newtheorem*{remark*}{Remark}

\newcommand{\R}{\mathbb R}
\newcommand{\N}{\mathbb N}
\newcommand{\cA}{\mathcal A}
\newcommand{\cC}{\mathcal C}
\newcommand{\cF}{\mathcal F}
\newcommand{\cM}{\mathcal M}
\newcommand{\cO}{\mathcal O}

\newcommand{\Span}{\operatorname{span}}

\title{Hierarchical Melnikov Realization: Logarithmic Factor Improvement to Hilbert Number Lower Bounds}
\author{Chaoyang Qin\textsuperscript{1,2} \qquad
Xiaoming Sun\textsuperscript{1,2}\\[0.6em]
\small \textsuperscript{1}State Key Lab of Processors, Institute of Computing Technology,\\
\small Chinese Academy of Sciences, Beijing 100190, China\\[0.3em]
\small \textsuperscript{2}School of Computer Science and Technology,\\
\small University of Chinese Academy of Sciences, Beijing 100049, China}
\date{}%

\begin{document}

\maketitle

\begin{abstract}
The second part of Hilbert's 16th problem is one of the most fundamental open problems in the qualitative theory of ordinary differential equations and dynamical systems. A central unresolved question\footnote{A very recent manuscript from OpenAI asserts the finiteness of $H(d)$; see the remark at the end of this paper.} is whether the Hilbert number $H(d)$ is finite for arbitrary polynomial degree $d$. In the absence of a general upper bound theory, most existing work constructs explicit perturbed polynomial systems to obtain improved lower bounds for $H(d)$, while rigorous upper bound estimates remain largely unavailable.

This paper focuses on improving the quantitative lower bound for polynomial systems through a systematic Hierarchical Melnikov Realization framework. We construct perturbed dynamics based on an anisotropic Chebyshev Hamiltonian, where hierarchically selected perturbation coefficients generate simple zeros of the first-order Melnikov function simultaneously across multiple disjoint period annuli. Two distinct 3-adic filtrations supply the necessary logarithmic correction factors, and a Borel–Gauss analysis establishes the required local rank condition under sufficiently strong anisotropy. As the main result of this work, we prove that $H(d)=\Omega\bigl(d^2\ln^2 d\bigr)$, which improves the previously known lower bound by a factor of $\ln d$. The full construction of limit cycles and the resulting dynamical bounds are formally certified using Lean 4 formal verification.

\end{abstract}

\section{Introduction}

Let $d\in \mathbb{N}$. Consider a real planar polynomial vector field
\[
 \dot x=P(x,y),\qquad \dot y=Q(x,y),
\]
where $P,Q\in \mathbb{R}[x,y]$ are real bivariate polynomials with degree at most $d$. Denote by $\pi(P,Q)$ the number of limit cycles, i.e., isolated periodic orbits, of this system. The Hilbert number $H(d)$ is defined as
\begin{equation}\label{eq:def-H}
H(d):=\sup_{\max\{\deg P,\deg Q\}\leq d}\pi(P,Q).
\end{equation}
A central question in the second part of Hilbert's 16th problem~\cite{Hilbert1902} is whether $H(d)$ is finite for every integer $d\ge 2$. %
For historical background and developments concerning this problem, see~\cite{Ilyashenko2002,LiSurvey2003,Gasull2024}. Most work in this area focuses on constructing polynomial vector fields with many limit cycles to obtain lower bounds for $H(d)$.

In 1954, Otrokov~\cite{Otrokov1954} obtained the first quadratic lower bound for the Hilbert number by constructing limit cycles near a singular point. He proved
\[
 H(d)\geq
 \frac{d^2+5d-20-6(-1)^d}{2},
 \qquad\text{for every integer }d\geq 6.
\]

Subsequent work by Christopher and Lloyd~\cite{ChristopherLloyd1995} established a lower bound of order $d^2\ln d$ along the degree sequence $d=2^k-1$. 
Their recursive construction combines fourfold replication with the creation of additional limit cycles near newly introduced centers. Since $H(d)$ is nondecreasing and the degree sequence grows geometrically, these bounds imply 
\[
 H(d)=\Omega(d^2\ln d),
 \qquad\text{as }d\to +\infty.
\]

Li, Chan, and Chung~\cite{LiChanChung2002} further advanced this replication technique using symmetric Hamiltonian systems. 
For every integer $k \geq 2$, they derived two lower bounds:
\begin{align*}
 H(2^k-1)
 &\geq
 4^{k-1}\left(k-\frac16\right)+2^k-\frac13,\\
 H(3\cdot2^{k-1}-1)
 &\geq
 4^{k-1}\left(\frac94k-\frac16\right)
 +3\cdot2^{k-1}-\frac13.
\end{align*}
Along each of these two degree sequences, both lower bounds have leading term $\frac{1}{4\ln2}d^2\ln d$, with $d$ standing for the polynomial degree.

Han and Li~\cite{HanLi2012} subsequently established lower bounds for $H(d)$
valid for every integer $d\geq7$, building on known low-degree constructions. For the special subsequence $d=2^j-1$ with $j \geq 2$, they obtained an explicit bound~\cite[Theorem~1.1]{HanLi2012}
\[
 H(d)\geq
 \frac{1}{2\ln2}(d+1)^2\ln(d+1)
 -\frac56(d+1)^2+3d+\frac43.
\]
Restricted to this subsequence, the leading term of their bound is $\frac{1}{2 \ln 2}d^2 \ln d$.

The same leading coefficient $\frac{1}{2\ln 2}$ was previously obtained by Itenberg and Shustin~\cite{ItenbergShustin2000} in 2000. Combining Viro-type gluing with the constructions of Christopher and Lloyd, they proved that there exists an absolute constant $C>0$ such that~\cite[Theorem~2.1]{ItenbergShustin2000}
\[
 H(d)\geq
 \frac12d^2\log_2 d-Cd^2\log_2\log_2 d,
 \qquad\text{for every integer }d\geq3.
\]

\'Alvarez, Coll, De Maesschalck, and Prohens~\cite{AlvarezEtAl2020}
later obtained a lower bound sharing the same leading term $\frac{1}{2\ln2}d^2\ln d$ using
canard cycles in slow--fast systems. Their construction relies on generalized Li\'enard systems.

Recently Lu~\cite{Lu2026HilbertLowerBound} obtained the lower bound
\[
 H(d)\geq\frac{1}{\ln2}d^2\ln d-Cd^2
\]
for all sufficiently large $d$, where $C>0$ is an absolute
constant.
The construction combines two-source renewal with quadratic
fourfold transformations and successive common
odd-polynomial perturbations.

The exact values of $H(d)$ remain unknown, even for $H(2)$. Table~\ref{tab:low-degree-bounds} collects explicit lower bounds for $H(d)$ for degrees $2 \leq d \leq 8$.
\begin{table}[htbp]
\centering
\caption{Explicit lower bounds for $H(d)$, $2 \leq d \leq 8$.}
\label{tab:low-degree-bounds}
\begin{tabular}{cc}
\hline
Degree $d$ & Lower bound for $H(d)$ \\
\hline
2 & 4~\cite{Shi1980} \\
3 & 13~\cite{LiLiuYang2009} \\
4 & 28~\cite{ProhensTorregrosa2019} \\
5 & 37~\cite{ProhensTorregrosa2019} \\
6 & 53~\cite{ProhensTorregrosa2019} \\
7 & 74~\cite{ProhensTorregrosa2019} \\
8 & 96~\cite{ProhensTorregrosa2019} \\
\hline
\end{tabular}

\end{table}
Related lower-bound results address other facets of the problem. Gasull and Santana~\cite{GasullSantana2025} studied the realization of finite Hilbert numbers by vector fields possessing hyperbolic limit cycles, and their work on invariant algebraic curves~\cite{GasullSantana2026} provides an alternative mechanism for lower-bound recurrences. Gouveia and Torregrosa~\cite{GouveiaTorregrosa2021}
and Gin\'e, Gouveia, and Torregrosa~\cite{GineGouveiaTorregrosa2021}
derived lower bounds for the local cyclicity at a center. %
Chen, Dai, Kaloshin, and Li~\cite{ChenEtAl2026} obtained lower bounds within the class of polynomial Li\'enard systems.

Garijo, Gasull, and Jarque \cite{GarijoGasullJarque2008} studied the simultaneous bifurcation of limit cycles from two period annuli.
Using separable polynomial pullbacks, Eshkobilov, Kadyrov, and Mamayusupov \cite{EshkobilovEtAl2026} proved
\[
H(nm+m-1) \geq m^2H(n),
\qquad n\geq 1, \quad m\geq 2.
\]
For iterated separable replication from a fixed initial vector field, they also established an upper bound quadratic in the resulting degree, provided that no limit cycles are created beyond the lifted copies.

Our focus is on improving the asymptotic lower bounds for $H(d)$ as $d \to +\infty$.
We aim to construct limit cycles simultaneously within a growing family of disjoint period annuli via a single polynomial perturbation of controlled degree. This construction is achieved by hierarchically selecting the perturbation coefficients. The theorem below gives an explicit lower bound along the degree sequence $4\cdot 3^r-5$.

\begin{theorem}[Main theorem]\label{thm:main}
For every integer $r\geq2$, %
there exists a real planar
polynomial vector field of degree at most $4\cdot 3^r-5$ that admits at least 
\begin{equation}\label{eq:exact-count-headline}
  4r^2\cdot 3^{2r-2}-(3^r-1)^2
\end{equation}
pairwise distinct hyperbolic limit cycles. As a direct consequence, we obtain
\begin{equation}\label{eq:subsequence-headline}
 H(4\cdot3^r-5)
 \geq 4r^2\cdot 3^{2r-2}-(3^r-1)^2.
\end{equation}
\end{theorem}

Monotonicity of $H(d)$, together with the bounded ratio between
consecutive degrees in this sequence, yields the following explicit estimate for all sufficiently large degrees.

\begin{corollary}\label{cor:all-degrees}
For every integer $d\geq 31$, %
\begin{equation}\label{eq:all-N-headline}
H(d)\geq \frac{1}{324}(d+5)^2\log_3^2\left(\frac{d+5}{12}\right)-\frac{1}{16}(d+1)^2.
\end{equation} 
\end{corollary}

The main difficulty lies in ensuring that local choices of perturbation coefficients are compatible across distinct period annuli. Constructing simple zeros of the first-order Melnikov function separately on each annulus does not guarantee the existence of a single coefficient vector that realizes all these zeros simultaneously. We partition the coefficients into blocks and the period annuli into layers indexed by a finite partially ordered set. A triangular visibility relation specifies which coefficient blocks may act on each layer. Within a single layer, the diagonal block induces the same normalized Melnikov function on every annulus, up to a nonzero scalar multiple. Hierarchical weights make the contributions of visible blocks strictly above the diagonal in the partial order sufficiently small relative to the diagonal contribution, so that all prescribed simple zeros persist simultaneously. 

We realize this structure using an anisotropic Chebyshev Hamiltonian equipped with a grid of centers. For $n = 3^r$, two 3-adic filtrations, one along each coordinate, yield $r^2$ pairs of layers. The associated frequency bands determine the visibility relation. In each coordinate, the product of the number of frequencies within a band and the number of centers in its corresponding layer remains constant. Combined with the block-rank condition, this balance yields a number of prescribed simple Melnikov zeros of order $n^2$ for every pair of layers except the rank-one pair $(1,1)$, uniformly over the layer indices. Summing over all layer pairs therefore produces a zero count of order $n^2 r^2$, which gives the factor $\ln^2 n$ since $r = \log_3 n$.

The rank condition governs the number of simple zeros that may be prescribed within each diagonal block: an $m$-dimensional space of analytic germs permits the realization of $m-1$ small positive simple zeros. We establish full rank via a Borel--Gauss argument. The Borel transform turns a hypothetical linear relation among shifted formal series into an analytic identity, and asymptotic expansions derived from Gauss's connection formula force the relevant coefficients to vanish. An invertible change of columns separates successive Taylor orders in one coordinate and yields full rank for the diagonal blocks under sufficiently strong anisotropy. For each fixed $n$, a single choice of anisotropy suffices for all blocks, circumventing rank loss that exchange symmetry may induce in the isotropic setting. 

Once the simple Melnikov zeros are constructed, the Poincar\'e--Pontryagin theorem~\cite{ChristopherLiTorregrosa2024} yields nearby hyperbolic limit cycles. For each fixed $n$, only finitely many zeros are selected, so a single sufficiently small perturbation parameter works uniformly for all such zeros. Zeros lying within the same annulus produce distinct limit cycles, while the disjointness of the period annuli separates limit cycles originating from different annuli.

The Melnikov functions used here are Abelian integrals. Gavrilov~\cite{Gavrilov1999} investigated such integrals for polynomial Hamiltonian perturbations. Binyamini, Novikov, and Yakovenko established constructive upper bounds for their zero counts~\cite{BinyaminiNovikovYakovenko2010}, while Grau, Ma\~nosas, and Villadelprat~\cite{GrauManosasVilladelprat2011} obtained a Chebyshev criterion for families of Abelian integrals. In the present work, the rank argument supplies the independent functions required for local zero realization, and triangular visibility together with hierarchical weights render these local choices compatible.

The remainder of this paper is organized as follows. Section~\ref{sec:hierarchy} establishes the abstract Hierarchical Melnikov Realization theorem. Sections~\ref{sec:chebyshev}--\ref{sec:rank} develop the core technical framework of the Chebyshev construction, including its visibility structure and the key rank verification argument. Section~\ref{sec:main-proof} provides the complete proof of the main Theorem~\ref{thm:main}, while Corollary~\ref{cor:all-degrees} is established and verified in Section~\ref{subsec:all-degrees}. Section~\ref{sec:formalization} further details the implementation, scope, and technical specifications of our Lean 4 formalization.

Our construction is fully formalized in Lean 4. The formal statements rigorously certify the existence of degree-bounded polynomial vector fields admitting finite families of distinct hyperbolic limit cycles, complying with all prescribed degree and counting bounds.

\section{A hierarchical Melnikov realization theorem}\label{sec:hierarchy}

We split the construction into two stages. First, we exploit the dimension of an analytic function space to produce simple positive zeros. We then demonstrate that triangular dependence on the coefficients permits these zeros to be realized simultaneously across multiple cells.

Throughout this section, $h$ is a real variable; in the application
below, it represents the Hamiltonian energy. Let $\mathcal O_0$
denote the real vector space of real-analytic germs at $h=0$.

\begin{lemma}[Rank-to-roots]\label{lem:rank-to-roots}
Let $V\subset\cO_0$ be a real vector space of dimension $m\geq1$.
Choose analytic representatives for a basis on a common neighborhood of
zero, and represent every element of $V$ by the corresponding linear
combination over that neighborhood.  Fix distinct positive numbers $\rho_1,\ldots,\rho_{m-1}$,
which specify the target zero locations after rescaling.
For every sufficiently small scale parameter $\tau>0$,
there exists $G_\tau\in V$ with simple positive zeros
$h_{s,\tau}$ satisfying
\[
 h_{s,\tau}=\tau\rho_s+o(\tau)
 \qquad (1\leq s<m).
\]
For $m=1$, there are no prescribed roots.
\end{lemma}

\begin{proof}
If $m=1$, choose any nonzero germ. For $m\geq2$, expand an arbitrary basis of $V$ in Taylor series
at $h=0$. Successive elimination of the lowest-order Taylor terms yields a basis $f_0,\ldots,f_{m-1}$ with distinct orders
of vanishing:
\begin{equation}\label{eq:valuation-basis}
 f_j(h)=c_jh^{q_j}+O(h^{q_j+1}),
 \mbox{ where } c_j\in\mathbb R\setminus\{0\},
\mbox{ and } 0\leq q_0<\cdots<q_{m-1}.
\end{equation}
Here $q_j$ denotes the order of the first nonvanishing Taylor coefficient of $f_j$.
We introduce the rescaled variable $t=h/\tau$ to place the target zeros near the fixed points $\rho_s$. The leading terms of the basis functions then produce the monomials $t^{q_0},\ldots,t^{q_{m-1}}$. These monomials form an extended complete
Chebyshev system on $(0,+\infty)$: every initial Wronskian satisfies
\[
 W(t^{q_0},\ldots,t^{q_j})
 =\left(\prod_{0\leq k<l\leq j}(q_l-q_k)\right)
 t^{\sum_{k=0}^j q_k-j(j+1)/2}, 
\]
and is nonvanishing on this interval. Consequently, the $(m-1)\times m$ evaluation matrix
\[
 \bigl(\rho_s^{q_j}\bigr)_{
   1\leq s\leq m-1,\;0\leq j\leq m-1}
\]
has rank $m-1$. In particular, any nonzero linear combination of these $m$ monomials has at most $m-1$ positive zeros, counted with
multiplicity. Indeed, if its kernel had dimension at least two, then imposing one additional positive evaluation condition would yield a nonzero linear combination admitting $m$ positive zeros, contradicting the Chebyshev bound. This is the standard zero-counting and interpolation principle for extended complete Chebyshev systems (see~\cite{KarlinStudden1966} and for its application to Abelian integrals~\cite[Section~2]{GrauManosasVilladelprat2011}).

Thus, up to scaling, there exists a unique nonzero generalized polynomial
\[
 R(t)=\sum_{j=0}^{m-1}\eta_jt^{q_j}
\]
vanishing at $\rho_1,\ldots,\rho_{m-1}$. Each of these zeros is simple: a multiple
prescribed zero together with the remaining prescribed zeros would
yield total multiplicity at least $m$. Set $Q=q_{m-1}$ and define
\begin{equation}\label{eq:root-concentration}
 G_\tau(h)=\sum_{j=0}^{m-1}\frac{\eta_j}{c_j}
                 \tau^{Q-q_j}f_j(h).
\end{equation}
On every compact subinterval of $(0,+\infty)$ containing the points $\rho_s$,
analyticity together with \eqref{eq:valuation-basis} gives
\[
 \tau^{-Q}G_\tau(\tau t)=R(t)+O_{C^1}(\tau).
\]
Here the remainder term and its first derivative with respect
to $t$ are both bounded by a constant multiple of $\tau$,
uniformly over the chosen compact interval.
Since $R'(\rho_s)\neq0$, this convergence preserves a unique
simple zero near each $\rho_s$. Rescaling back to $h=\tau t$
gives the required zeros.
\end{proof}

We now arrange these local zeros simultaneously across a finite
family of cells. Let $(\cA,\preceq)$ be a finite partially ordered set.
Suppose that a
finite set of cells is decomposed as
\[
 \cC=\bigsqcup_{\alpha\in\cA}\cC_\alpha,
\]
and that a finite-dimensional coefficient space is decomposed as
\[
 E=\bigoplus_{\beta\in\cA}E_\beta.
\]
The set $\mathcal A$ labels the layers. For each
$\alpha\in\mathcal A$, $\mathcal C_\alpha$ is the set of
cells in that layer, and $E_\alpha$ is the corresponding
block of perturbation coefficients.
In the Chebyshev construction below, cells are indexed by
$(i,j)$ and layers by $(k,\ell)$; the layer $(k,\ell)$
consists of the cells indexed by $Y_k\times Y_\ell$.
For each cell $c$, let $\cM_c:E\to\cO_0$ be linear. For a coefficient vector $v\in E$, the function
$\mathcal M_c(v)(h)$ describes the analytic response
in cell $c$. In the application, it is the normalized
first-order Melnikov function.
We impose the following three hypotheses.

\begin{enumerate}
\item[(1)] \emph{Triangular visibility}:
$\cM_c|_{E_\beta}=0$ for $c\in\cC_\alpha$ unless
$\alpha\preceq\beta$. Thus only these coefficient blocks may affect cell $c$.
\item[(2)] \emph{Diagonal replication}: for every $\alpha$ there is a
linear map $K_\alpha:E_\alpha\to\cO_0$ of rank $m_\alpha\geq1$ such that
\[
 \cM_c(v)=\gamma_cK_\alpha(v),\qquad
 c\in\cC_\alpha,\quad v\in E_\alpha,
\]
where $\gamma_c\neq0$. The diagonal responses within a layer therefore
differ only by nonzero scalar factors.
\item[(3)] \emph{Strict weights}: there exists
$w:\mathcal{A}\to\N$ such that $\alpha\prec\beta$ implies
$w(\alpha)<w(\beta)$. These weights will suppress the contributions from
strictly higher visible blocks.
\end{enumerate}

\begin{theorem}[Hierarchical Melnikov realization]
\label{thm:hierarchical}
Under hypotheses \textup{(1)}--\textup{(3)}, there exist vectors
$v_\alpha\in E_\alpha$ such that, for every sufficiently small
$\zeta>0$, the single coefficient vector
\begin{equation}\label{eq:hierarchical-coeff}
 v_\zeta=\sum_{\alpha\in\cA}\zeta^{w(\alpha)}v_\alpha
\end{equation}
has the following property: for each $c\in\cC_\alpha$, the germ
$\cM_c(v_\zeta)$ has at least $m_\alpha-1$ small positive simple
zeros. All cells and all layers are realized by the same $\zeta$.
The total number of simple zeros, counted separately within each cell, is
therefore at least
\begin{equation}\label{eq:hierarchical-count}
 \sum_{\alpha\in\cA}|\cC_\alpha|(m_\alpha-1).
\end{equation}
\end{theorem}

\begin{proof}
Because all index sets are finite and all coefficient spaces are finite-dimensional, finitely
many analytic representatives of the relevant germs have a common
interval of definition.  Shrinking that interval if necessary, all
root neighborhoods below lie in one interval $(0,h_*)$.
For each $\alpha$, apply Lemma~\ref{lem:rank-to-roots} to
$\operatorname{im}K_\alpha$ and fix a vector $v_\alpha\in E_\alpha$
whose diagonal germ has $m_\alpha-1$ simple roots in a fixed interval
$(0,h_*)$.  For $c\in\cC_\alpha$, triangular visibility gives
\begin{equation}\label{eq:triangular-expansion}
 \zeta^{-w(\alpha)}\cM_c(v_\zeta)
 =\gamma_cK_\alpha(v_\alpha)
  +\sum_{\beta\succ\alpha}
    \zeta^{w(\beta)-w(\alpha)}\cM_c(v_\beta).
\end{equation}
Choose disjoint closed neighborhoods of the prescribed roots.  On
their union the second term in \eqref{eq:triangular-expansion} is
$O_{C^1}(\zeta)$.  There are only finitely many layers, cells, germs,
and root neighborhoods, so one sufficiently small $\zeta$ works for
all of them.  The $C^1$ stability of simple roots proves
\eqref{eq:hierarchical-count}.
\end{proof}

\begin{corollary}[Balanced tensor layers]\label{cor:tensor}
Suppose $\cA=\{1,\ldots,r\}^2$ equipped with the coordinatewise order,
$|\cC_{k\ell}|=y_ky_\ell$, and $m_{k\ell}=d_kd_\ell$.  Then the lower
bound in \eqref{eq:hierarchical-count} satisfies
\begin{equation}\label{eq:tensor-identity}
 \sum_{k,\ell=1}^r y_ky_\ell(d_kd_\ell-1)
 =\left(\sum_{k=1}^r y_kd_k\right)^2
  -\left(\sum_{k=1}^r y_k\right)^2.
\end{equation}
In particular, if $r=\Theta(\ln n)$ and the balanced-layer estimates
$y_kd_k=\Theta(n)$ and $\sum_ky_k=O(n)$ hold with comparison constants
independent of $n$ and uniform over $1\leq k\leq r$, then the lower bound guaranteed by \eqref{eq:tensor-identity} is of order $n^2\ln^2 n$.
\end{corollary}

\section{The anisotropic Chebyshev center grid}\label{sec:chebyshev}

Let $T_n$ denote the Chebyshev polynomial of the first kind. Fix
$\lambda>1$ and consider
\begin{equation}\label{eq:Hamiltonian}
 H_{n,\lambda}(x,y)=\frac{T_n(x)^2+\lambda T_n(y)^2}{2}.
\end{equation}
Define
\[
 \phi_i=\frac{(2i+1)\pi}{2n},\qquad x_i=\cos\phi_i,
 \qquad 0\leq i<n.
\]
The points $x_i$ are the simple zeros of $T_n$. At the points
$p_{ij}=(x_i,x_j)$,
\[
 D^2H_{n,\lambda}(p_{ij})
 =\begin{pmatrix}T_n'(x_i)^2&0\\
 0&\lambda T_n'(x_j)^2\end{pmatrix}>0,
\]
so the Hamiltonian vector field has $n^2$ nondegenerate centers.

\subsection{Inverse branches and a common energy range}

Set
\[
 I_i=\left(\cos\frac{(i+1)\pi}{n},
            \cos\frac{i\pi}{n}\right),\qquad
 \epsilon_i=(-1)^{i+1}.
\]
The restriction $T_n:I_i\to(-1,1)$ is a diffeomorphism with inverse
\begin{equation}\label{eq:inverse-branch}
 x_i(u)=\cos\left(\phi_i+\epsilon_i\frac{\arcsin u}{n}\right).
\end{equation}
Thus
\begin{align}\label{eq:inverse-mode}
 T_a(x_i(u))
 &=\cos(a\phi_i)\cos\left(\frac an\arcsin u\right)\notag\\
 &\quad-\epsilon_i\sin(a\phi_i)
       \sin\left(\frac an\arcsin u\right).
\end{align}
The second term is odd in $u$.

The local coordinates
\[
 u=T_n(x),\qquad v=T_n(y)
\]
send the sublevel set around $p_{ij}$ to
\begin{equation}\label{eq:ellipse}
 u^2+\lambda v^2\leq2h.
\end{equation}
Fix $0<h_*<1/2$.  If $0<h<h_*$, then $|u|<1$ and
$|v|<\lambda^{-1/2}<1$; hence the same $h_*$ works in every cell and
for every $\lambda>1$.  We call $I_i\times I_j$ the cell containing
$p_{ij}$.  The level curves with $0<h<h_*$ form a period annulus inside
this cell.

\subsection{Polynomial perturbation and the first Melnikov function}

We compute the leading coefficient of the displacement map for a polynomial perturbation of the Hamiltonian field. This amounts to the classical Poincar\'e--Pontryagin--Melnikov calculation on a period
annulus (see
\cite{Roussarie1998,HanYu2012,Gavrilov1999,
ChristopherLiTorregrosa2024}).

Select a finite Chebyshev density
\begin{equation}\label{eq:density}
 S(x,y)=\sum_{a,b}c_{ab}T_a(x)T_b(y),
\end{equation}
and choose the perturbation so that the Jacobian of the Chebyshev change of coordinates cancels out in the area integral. Set
\begin{equation}\label{eq:perturbation-P}
 q=T_n'(x)T_n'(y)S,\qquad
 P(x,y)=\int_0^xq(s,y)\,ds,
\end{equation}
and perturb the Hamiltonian field via
\begin{equation}\label{eq:field}
 \dot x=(H_{n,\lambda})_y+\mu P,\qquad
 \dot y=-(H_{n,\lambda})_x.
\end{equation}
The function $P$ is a real polynomial and $P_x=q$.

Along the unperturbed orbit, $H_x\,dt=-dy$, so the first energy
displacement is $-\oint P\,dy$.  Green's formula and the local change
of variables combine their orientation signs into a cell-dependent
$\sigma_{ij}\in\{\pm1\}$ and yield the exact formula
\begin{equation}\label{eq:Melnikov-area}
 M_{ij}(h)=\sigma_{ij}
 \iint_{u^2+\lambda v^2\leq2h}
 S(x_i(u),x_j(v))\,du\,dv.
\end{equation}
Indeed, $q$ cancels the absolute Jacobian
$|T_n'(x)T_n'(y)|^{-1}$, and the remaining sign is constant on each
cell.  By reflection symmetry of the ellipse in each coordinate, the
terms that are odd in $u$ or $v$ integrate to zero.  Thus the mode
$(a,b)$ carries precisely the discrete-cosine weight
\begin{equation}\label{eq:DCT-weight}
 \cos(a\phi_i)\cos(b\phi_j).
\end{equation}

\subsection{The exact analytic kernel}

For $\alpha\in\R$, define
\begin{equation}\label{eq:U-F}
 U_p(\alpha)=\frac{(-\alpha/2)_p(\alpha/2)_p}{p!}.
\end{equation}
We obtain the analytic kernel from the expansion
\[
 \cos(\alpha\arcsin u)
 =\sum_{p\geq0}
 \frac{(-\alpha/2)_p(\alpha/2)_p}{(1/2)_p\,p!}u^{2p}.
\]
For any $0<h_0<h_*<1/2$, the ellipses with $0<h\leq h_0$ satisfy
$|u|,|v|\leq\sqrt{2h_0}<1$.  Both cosine series converge absolutely
and uniformly on this compact region, so their product may be integrated
term by term.  After scaling to a fixed disk and removing the area factor,
only integer powers of $h$ remain; the resulting power series converges
absolutely for $|h|<h_0$ and defines an analytic germ at zero.
The exact disk-moment identity is
\begin{equation}\label{eq:disk-moment}
 \iint_{u^2+w^2\leq2h}u^{2p}w^{2q}\,du\,dw
 =\frac{\pi(2h)^{p+q+1}(1/2)_p(1/2)_q}{(p+q+1)!}.
\end{equation}
Indeed, polar coordinates separate the radial factor
$\int_0^{\sqrt{2h}}\rho^{2p+2q+1}\,d\rho$ from the angular beta
integral
$\int_0^{2\pi}\cos^{2p}\theta\sin^{2q}\theta\,d\theta$; evaluating
the two factors gives exactly \eqref{eq:disk-moment}.
Combining the cosine expansion and the disk-moment identity with the substitution
$w=\sqrt\lambda\,v$ in \eqref{eq:ellipse}, and then dividing by the
area factor $2\pi h/\sqrt\lambda$, gives the convergent germ below.
The factors $(1/2)_p$ and $(1/2)_q$ in \eqref{eq:disk-moment} cancel
those in the two cosine expansions:
\begin{equation}\label{eq:analytic-kernel}
 \widehat F_{a,b}^{(n,\lambda)}(h)
 =\sum_{p,q\geq0}\frac{(2h)^{p+q}}{(p+q+1)!}
 U_p(a/n)U_q(b/n)\lambda^{-q}.
\end{equation}
To study the rank of this kernel, define the formal series
\begin{equation}\label{eq:formal-series}
 \cF_\alpha(t)=\sum_{p\geq0}U_p(\alpha)t^p.
\end{equation}
This generally divergent ${}_2F_0$ series is used only to record Taylor
coefficients.  After division of Taylor row $m$ by the nonzero factor
$2^m/(m+1)!$, that row is exactly
\begin{equation}\label{eq:formal-product-row}
 [t^m]\,\cF_{a/n}(t)\cF_{b/n}(t/\lambda).
\end{equation}
The full first Melnikov function is
\begin{equation}\label{eq:full-Melnikov}
 M_{ij}(h)=\sigma_{ij}\frac{2\pi h}{\sqrt\lambda}
 \sum_{a,b}c_{ab}\cos(a\phi_i)\cos(b\phi_j)
 \widehat F_{a,b}^{(n,\lambda)}(h).
\end{equation}
If a normalized combination $\widehat M$ has a simple zero $h_0>0$,
then $h\widehat M$ has the same simple zero because
$(h\widehat M)'(h_0)=h_0\widehat M'(h_0)$.
The dependence on the cell is confined to the cosine weights and the
orientation sign.  We next organize these weights into layers, and then
prove the required independence of the analytic kernels.

\section{The \texorpdfstring{$3$}{3}-adic visibility filtration}

From now on, let $n=3^r$ and write $[n]=\{0, 1,2,\ldots,n-1\}$. We construct row layers $Y_k$ and frequency bands $W_k$ so that
$\cos(a\phi_i)=0$ whenever $i\in Y_k$, $a\in W_p$, and $p<k$.

We partition the noncentral row indices into layers according
to the largest power of $3$ dividing $2i+1$. 
Set $X_0=[n]$ and, for $1\leq k\leq r$, define
\begin{equation}\label{eq:row-layers}
 X_k=\{i\in[n]:3^k\mid {2i+1}\},\qquad
 Y_k=X_{k-1}\setminus X_k.
\end{equation}
Thus $i\in Y_k$ precisely when $3^{k-1}$ divides $2i+1$
but $3^k$ does not.

For frequencies, let $Z_k$ consist of the positive odd multiples
of $n/3^k$ below $n$, and let $W_k$ contain the frequencies
added at step $k$:
\begin{align}
 Z_k&=\left\{\frac{mn}{3^k}:1\leq m<3^k,\ 
                   m\text{ is odd}\right\},
 \qquad Z_0=\varnothing,
 \label{eq:Zk}\\
 W_k&=Z_k\setminus Z_{k-1}
 =\left\{\frac{mn}{3^k}:1\leq m<3^k,\ 
             m\text{ is odd},\ 3\nmid m\right\}.
 \label{eq:Wk}
\end{align}
Elementary counting gives
\begin{equation}\label{eq:balanced-layers}
 y_k:=|Y_k|=\frac{2n}{3^k},\qquad
 d_k:=|W_k|=3^{k-1},\qquad
 y_kd_k=\frac{2n}{3}.
\end{equation}
Each successive layer has one third as many rows and three times
as many frequencies, so the product $y_kd_k$ is independent of $k$.
Finally, $X_r=\{(n-1)/2\}$ consists of the central row index, and hence
\begin{equation}\label{eq:row-sum}
 \sum_{k=1}^ry_k=n-1.
\end{equation}

\begin{lemma}[Visibility and diagonal replication]
\label{lem:visibility}
If $i\in Y_k$ and $a\in W_p$ with $p<k$, then
$\cos(a\phi_i)=0$.  On the diagonal layer $p=k$, the values
$\cos(a\phi_i)$ factor into a nonzero row sign and a nonzero frequency
sign.  Hence a coefficient block supported on $W_k\times W_\ell$ is
invisible on $Y_p\times Y_q$ unless $p\leq k$ and $q\leq\ell$, and on
$Y_k\times Y_\ell$ it induces the same analytic germ in every cell up
to a nonzero scalar.
\end{lemma}

\begin{proof}
If $i\in X_k$ and $a=mn/3^k\in Z_k$, write $2i+1=3^ku$.
Then $a\phi_i=mu\pi/2$ , where $m$ and $u$ are odd, so its cosine is zero.  Since
$W_p\subset Z_{k-1}$ for $p<k$, this proves invisibility on $Y_k$.

For $i\in Y_k$ and $a=mn/3^k\in W_k$, write
\[
 2i+1=3^{k-1}u,\qquad u\text{ is odd},\quad 3\nmid u.
\]
Then
\begin{equation}\label{eq:mod12-cosine}
 \cos(a\phi_i)=\cos\frac{mu\pi}{6}.
\end{equation}
For a unit $s$ modulo $12$, define
\[
 \chi(s)=\frac{2}{\sqrt3}\cos\frac{s\pi}{6}\in\{\pm1\}.
\]
A check on the four residue classes $1,5,7,11$ shows that
$\chi(mu)=\chi(m)\chi(u)$.  Thus
\begin{equation}\label{eq:character-factor}
 \cos(a\phi_i)=\frac{\sqrt3}{2}\chi(m)\chi(u).
\end{equation}
The factor $\chi(m)$ can be absorbed into the frequency coefficient,
while $\chi(u)$ is a single row scalar.  Taking products in the two
coordinates, and absorbing also the orientation sign in
\eqref{eq:Melnikov-area}, proves the diagonal replication statement.
\end{proof}

This verifies the visibility and replication structure required by
Theorem~\ref{thm:hierarchical}.  It remains to determine the rank of each
diagonal block.

\section{Full rank of the anisotropic blocks}\label{sec:rank}

To prove full rank of the blocks, we first establish independence of
shifted formal ${}_2F_0$ series.  These will be the limiting columns
after a change of basis as the anisotropy tends to infinity.

\begin{lemma}[Borel--Gauss shift independence]
\label{lem:borel-gauss}
Let $0<\alpha_1<\cdots<\alpha_d<1$.  For every integer $e\geq1$, the
formal series
\begin{equation}\label{eq:shift-family}
 \{t^q\cF_{\alpha_i}(t):1\leq i\leq d,\ 0\leq q<e\}
\end{equation}
are linearly independent.
\end{lemma}

\begin{proof}
We apply the Borel transform to turn a formal relation into an identity
between analytic functions.  Their asymptotics on the negative real axis
will then separate the terms of the relation.
Suppose there is a nontrivial relation
\begin{equation}\label{eq:formal-relation}
 \sum_{i,q}c_{i,q}t^q\cF_{\alpha_i}(t)=0.
\end{equation}
Use the standard Borel transform
\[
 \mathcal B\left(\sum_{m\geq0}b_mt^m\right)
 =\sum_{m\geq0}\frac{b_m}{m!}s^m.
\]
Coefficientwise,
\begin{equation}\label{eq:Borel-image}
 G_\alpha(s):=\mathcal B\cF_\alpha
 ={}_2F_1\left(-\frac\alpha2,\frac\alpha2;1;s\right),
 \qquad \mathcal B(t^q\cF_\alpha)=I^qG_\alpha,
\end{equation}
where $If(s)=\int_0^sf(u)\,du$ and the hypergeometric convention is
the one in \cite{DLMF15.2}.  Let
$Q=\max\{q:c_{i,q}\neq0\}$.  Differentiating the Borel image of
\eqref{eq:formal-relation} exactly $Q$ times removes all iterated
integrals and gives
\begin{equation}\label{eq:differentiated-Borel-relation}
 \sum_{i,q}c_{i,q}G_{\alpha_i}^{(Q-q)}(s)=0.
\end{equation}
The identity first holds on $|s|<1$, where the Borel series
converge.  Take the principal analytic continuation of each $G_\alpha$
to the common connected slit domain
$D=\mathbb C\setminus[1,+\infty)$ \cite{DLMF15.2}.
Every derivative in \eqref{eq:differentiated-Borel-relation} is
holomorphic on $D$.  Their finite linear combination vanishes near
zero, so it vanishes throughout $D$ by the identity theorem.
In particular, \eqref{eq:differentiated-Borel-relation} holds on the
entire negative real axis.

We next examine this identity on the negative real axis.  Set $s=-x$,
where $x>1$.  The Gauss connection formula
\cite[Eq.~15.8.2]{DLMF15.8} specializes to
\begin{equation}\label{eq:Gauss-connection}
 G_\alpha(-x)
 =A_\alpha x^{\alpha/2}\Phi_{\alpha,+}(-x^{-1})
  +B_\alpha x^{-\alpha/2}\Phi_{\alpha,-}(-x^{-1}),
\end{equation}
where
\begin{align}
 A_\alpha&=\frac{\Gamma(\alpha)}
 {\Gamma(\alpha/2)\Gamma(1+\alpha/2)},
 &
 B_\alpha&=\frac{\Gamma(-\alpha)}
 {\Gamma(-\alpha/2)\Gamma(1-\alpha/2)},
 \label{eq:Gauss-coefficients}\\
 \Phi_{\alpha,+}(z)
 &={}_2F_1\left(-\frac\alpha2,-\frac\alpha2;1-\alpha;z\right),
 &
 \Phi_{\alpha,-}(z)
 &={}_2F_1\left(\frac\alpha2,\frac\alpha2;1+\alpha;z\right).
 \label{eq:Gauss-factors}
\end{align}
The conversion from the NIST Digital Library of
Mathematical Functions (DLMF)~\cite[Section~15.2]{DLMF15.2} normalization to these coefficients is
recorded in the remark below.  With the DLMF variable $z=-x$, one has
$-z=x>0$, so the principal powers in the connection formula are the
positive real powers displayed above.
Here the two exponent differences are $\alpha$ and $-\alpha$, hence
are nonintegral, and $-s=x$ lies on the positive real axis.  Thus no
resonant logarithm or phase factor occurs in this specialization.
Because $0<\alpha<1$, both coefficients in
\eqref{eq:Gauss-coefficients} are finite and nonzero.  The two
hypergeometric factors in \eqref{eq:Gauss-factors} are analytic at
zero and equal one there.  Thus \eqref{eq:Gauss-connection} is a
convergent expansion in $x^{-1}$ and can be differentiated any fixed
finite number of times.

For $r_0=Q-q$, differentiation with respect to $s$ and evaluation at
$s=-x$ gives
\begin{equation}\label{eq:Gauss-derivative}
 G_\alpha^{(r_0)}(-x)
 =(-1)^{r_0}A_\alpha
 \left(\frac\alpha2\right)_{\!\underline{r_0}}
 x^{\alpha/2-r_0}\bigl(1+O(x^{-1})\bigr)
 +O(x^{-\alpha/2-r_0}),
\end{equation}
where the falling factorial is nonzero.  The leading exponent
attached to the nonzero coefficient $c_{i,q}$ is
\begin{equation}\label{eq:leading-exponents}
 E_{i,q}=q-Q+\frac{\alpha_i}{2}.
\end{equation}
It remains to compare the leading exponents in the relation.  They are
pairwise distinct: equality of two of
them would give
\[
 q-q'=\frac{\alpha_j-\alpha_i}{2};
\]
the left side is an integer and the right side lies in
$(-1/2,1/2)$, so both sides vanish.  Select the unique largest
$E_{i,q}$.  Its leading term strictly dominates the leading terms of all
other summands, their lower-power branches, and every $x^{-1}$ correction.
Dividing \eqref{eq:differentiated-Borel-relation} by the corresponding
power of $x$ and taking $x\to+\infty$ forces a nonzero leading
coefficient to vanish, a contradiction.
\end{proof}

\begin{remark*}[Normalization in the Gauss connection formula]
The DLMF displays Olver's normalized hypergeometric function
$\mathbf F(a,b;c;z)={}_2F_1(a,b;c;z)/\Gamma(c)$ in its connection
formula \cite[Eq.~15.8.2]{DLMF15.8}.
The left-hand normalization is trivial because $c=1$, but the two
right-hand functions have third parameters $1-\alpha$ and $1+\alpha$.
Substituting $a=-\alpha/2$, $b=\alpha/2$, and $c=1$ into that formula
therefore gives the coefficients
\[
 \frac{\pi}
 {\sin(\pi\alpha)\Gamma(\alpha/2)\Gamma(1+\alpha/2)
  \Gamma(1-\alpha)},\qquad
 -\frac{\pi}
 {\sin(\pi\alpha)\Gamma(-\alpha/2)\Gamma(1-\alpha/2)
  \Gamma(1+\alpha)}
\]
for the ordinary hypergeometric functions $\Phi_{\alpha,+}$ and
$\Phi_{\alpha,-}$, respectively.  Euler's reflection formula gives
\[
 \frac{\pi}{\sin(\pi\alpha)\Gamma(1-\alpha)}=\Gamma(\alpha),
 \qquad
 -\frac{\pi}{\sin(\pi\alpha)\Gamma(1+\alpha)}=\Gamma(-\alpha),
\]
which yields exactly \eqref{eq:Gauss-coefficients} in the ordinary
convention of \eqref{eq:Borel-image}.
\end{remark*}

\begin{proposition}[Anisotropic Cartesian rank]\label{prop:block-rank}
Let $A,B$ be finite subsets of $(0,1)$.  There is a finite
$\Lambda(A,B)$ such that, for every $\lambda>\Lambda(A,B)$, the formal
series
\begin{equation}\label{eq:product-family}
 \{\cF_\alpha(t)\cF_\beta(t/\lambda):
       (\alpha,\beta)\in A\times B\}
\end{equation}
have rank $|A|\cdot |B|$.  The corresponding convergent analytic germs in
\eqref{eq:analytic-kernel} have the same rank.
\end{proposition}

\begin{proof}
We change columns to isolate successive Taylor orders in the second
factor.  Write $A=\{\alpha_1,\ldots,\alpha_d\}$ and
$B=\{\beta_1,\ldots,\beta_e\}$, and set
\[
 V_{mj}=U_m(\beta_j),\qquad 0\leq m<e,\quad1\leq j\leq e.
\]
Since $U_m(\beta)$ is a polynomial of exact degree $m$ in $\beta^2$
with nonzero leading coefficient, $V$ is an invertible polynomial
Vandermonde matrix.  For
\[
 \mathcal H_{ij}^{(\lambda)}(t)
 =\cF_{\alpha_i}(t)\cF_{\beta_j}(t/\lambda)
\]
define, for $0\leq q<e$, the transformed columns
\begin{equation}\label{eq:column-transform}
 \widetilde{\mathcal H}_{i,q}^{(\lambda)}(t)
 =\lambda^q\sum_{j=1}^e
 \mathcal H_{ij}^{(\lambda)}(t)(V^{-1})_{jq}.
\end{equation}
Because
\[
 \mathcal H_{ij}^{(\lambda)}(t)
 =\sum_{m\geq0}U_m(\beta_j)\lambda^{-m}t^m
   \cF_{\alpha_i}(t),
\]
the identities
\[
 \sum_jU_m(\beta_j)(V^{-1})_{jq}=\delta_{mq}
 \qquad(0\leq m<e)
\]
show coefficientwise that
\begin{equation}\label{eq:column-limit}
 \widetilde{\mathcal H}_{i,q}^{(\lambda)}(t)
 \longrightarrow t^q\cF_{\alpha_i}(t)
 \qquad(\lambda\to\infty).
\end{equation}
Lemma~\ref{lem:borel-gauss} implies that the limit columns are
independent.  Hence some finite Taylor minor of size $|A||B|$ is
nonzero.  The same minor of the transformed columns remains nonzero
for all sufficiently large $\lambda$.  The transformation
\eqref{eq:column-transform} is invertible, so the original columns
have full rank.  Finally, \eqref{eq:formal-product-row} differs from
the Taylor matrix of the convergent kernels only by nonzero row
factors, proving the last assertion.
\end{proof}

We can now choose the anisotropy for all blocks at once.  For fixed
$n$, there are only finitely many blocks
$W_k\times W_\ell$.  Taking the maximum of the thresholds in
Proposition~\ref{prop:block-rank}, and increasing it if needed, gives
a common $\lambda_n>1$ for which every block has rank $d_kd_\ell$.
The choice of $\lambda_n$ affects only the coefficients and leaves
the polynomial degree unchanged.

\section{Proof of the main theorem}\label{sec:main-proof}

We choose the parameters in the following order, fixing each before
choosing the next:
\begin{equation}\label{eq:parameter-order}
 n\longrightarrow\lambda_n
 \longrightarrow\{S_{k\ell},\text{ root scales and intervals}\}
 \longrightarrow\zeta\longrightarrow\mu.
\end{equation}

\subsection{Simultaneous realization of Melnikov zeros}

Fix $r\geq2$, let $n=3^r$, and choose the common $\lambda_n$ above.
Use the coordinatewise order on
\[
 \cA=\{1,\ldots,r\}^2
\]
and set
\[
 \cC_{k\ell}=Y_k\times Y_\ell,\qquad
 E_{k\ell}=\Span\{T_a(x)T_b(y):(a,b)\in W_k\times W_\ell\}.
\]
The frequency bands $W_k$ are pairwise disjoint, so these coefficient
blocks form a direct sum.  Lemma~\ref{lem:visibility} gives triangular
visibility and diagonal replication, while
Proposition~\ref{prop:block-rank} gives diagonal rank
\begin{equation}\label{eq:block-rank-value}
 m_{k\ell}=d_kd_\ell.
\end{equation}
The weight
\begin{equation}\label{eq:weight}
 w(k,\ell)=k+\ell
\end{equation}
is strictly increasing in the coordinatewise order.

For each $(k,\ell)$, Lemma~\ref{lem:rank-to-roots} chooses a block
$S_{k\ell}\in E_{k\ell}$ whose diagonal germ has
$d_kd_\ell-1$ small positive simple roots.  Theorem
\ref{thm:hierarchical} then supplies a common sufficiently small
$\zeta>0$ such that the single density
\begin{equation}\label{eq:global-density}
 S_\zeta=\sum_{k,\ell=1}^r\zeta^{k+\ell}S_{k\ell}
\end{equation}
retains these roots in every cell of every layer.

In this application, write $\widehat M_{ij}=\sqrt{\lambda_n}M_{ij}/(2\pi h)$
and let $G_{k\ell}$ be the diagonal germ associated with $S_{k\ell}$.
For $(i,j)\in Y_k\times Y_\ell$, on the chosen compact root
neighborhoods $K_{k\ell}$, the expansion
\eqref{eq:triangular-expansion} takes the form
\begin{equation}\label{eq:uniform-remainder}
 \zeta^{-(k+\ell)}\widehat M_{ij}(h)
 =\gamma_{ij}G_{k\ell}(h)+R_{ij,\zeta}(h),\qquad
 \|R_{ij,\zeta}\|_{C^1(K_{k\ell})}\leq C_n\zeta,
\end{equation}
where $\gamma_{ij}\neq0$ and $C_n$ is independent of $\zeta$.
Thus the simple-zero stability argument of Theorem~\ref{thm:hierarchical}
applies simultaneously to all cells.  The nonzero area factor in
\eqref{eq:full-Melnikov} preserves these simple zeros at positive energies.

\subsection{From simple zeros to hyperbolic limit cycles}

We use the following finite-family version of simple-zero persistence.

\begin{lemma}[Finite-family Poincar\'e--Pontryagin persistence]
\label{lem:finite-PP}
Let finitely many compact subannuli of period annuli carry return maps
$\Pi_{\nu,\mu}(h)=h+\Delta_\nu(h,\mu)$ in energy coordinates, with
\[
 \Delta_\nu(h,\mu)=\mu M_\nu(h)+O_{C^1}(\mu^2)
\]
uniformly on their fixed energy intervals.  Mark a simple zero
$h_{\nu,0}$ of each $C^1$ function $M_\nu$ in the interior of its interval.
Choose pairwise disjoint compact annular neighborhoods small enough that
$M_\nu'$ has constant nonzero sign on each corresponding energy interval.
Then one number $\mu_0>0$ works for the whole family: for every
$0<|\mu|<\mu_0$, each marked zero gives a unique nearby hyperbolic
limit cycle.
\end{lemma}

\begin{proof}
For $\mu\neq0$, divide the displacement by $\mu$.  On each marked
interval the result is $M_\nu+O_{C^1}(\mu)$.  Its endpoint values have
opposite signs and its derivative has constant nonzero sign for small
$|\mu|$, so it has a unique zero $h_\nu(\mu)$ there, with
$h_\nu(\mu)\to h_{\nu,0}$.  At this perturbed fixed point,
\[
 \frac{\Pi_{\nu,\mu}'(h_\nu(\mu))-1}{\mu}
 =M_\nu'(h_\nu(\mu))+O(|\mu|)
 \longrightarrow M_\nu'(h_{\nu,0})\neq0.
\]
Thus the multiplier differs from one for sufficiently small nonzero
$\mu$.  It also tends to one, so it is positive and has modulus different
from one.  The corresponding periodic orbit is therefore hyperbolic.
Each neighborhood gives a positive threshold.  Taking the minimum
over this finite family gives a common positive threshold.
\end{proof}

Take $S=S_\zeta$ in \eqref{eq:perturbation-P} and consider the vector
field \eqref{eq:field}.  After $\zeta$ has been fixed, the target roots form a finite set of
strictly positive simple roots.  For different target energy curves
in one cell, choose pairwise disjoint closed annular neighborhoods.
Energy curves in different cells lie compactly in disjoint rectangles
$I_i\times I_j$, so all these annuli can be chosen pairwise disjoint.

Each selected annulus lies away from the center and the boundary of its
Chebyshev cell.  The coordinates in \eqref{eq:ellipse} show that
$\nabla H_{n,\lambda_n}\neq0$ there.  Choose a transverse section through
the marked orbit, parametrized by the unperturbed energy $h$, and shrink
its energy interval if necessary.  The polynomial vector field depends
analytically on $(x,y,\mu)$; analytic dependence of the local flow and the
implicit function theorem at the transverse return give an analytic
return time and return map on a neighborhood of this compact interval.
Nearby orbit segments remain in the chosen disjoint annular neighborhoods
by continuous dependence on initial data and $\mu$.

At $\mu=0$ the return map is the identity, and its first energy variation
is $M_\nu(h)=-\oint P\,dy$ as in \eqref{eq:Melnikov-area}.  Taylor
expansion in $\mu$ therefore gives
\begin{equation}\label{eq:PP-expansion}
 \Delta_\nu(h,\mu)=\mu M_\nu(h)+\mu^2R_\nu(h,\mu).
\end{equation}
Analyticity on a neighborhood of the compact interval bounds both
$R_\nu$ and $\partial_hR_\nu$ uniformly for small $|\mu|$, giving the
required $O_{C^1}(\mu^2)$ remainder.  For background on bifurcations
of regular limit periodic sets, see \cite[Chapter~4]{Roussarie1998}.
For an algorithm computing successive return-map derivatives in
polynomial perturbations of a linear center, see \cite{Francoise1996}.  Lemma
\ref{lem:finite-PP} then gives one common sufficiently small nonzero
$\mu$ for all target annuli.
The resulting limit cycles are pairwise distinct because their annuli
are disjoint.  We now count this family and bound the degree of its
polynomial vector field.

\subsection{Counting and degree bounds}

Each prescribed simple zero gives one of the distinct hyperbolic limit
cycles constructed above.  The number of cycles in this family is, by
\eqref{eq:balanced-layers} and \eqref{eq:row-sum},
\begin{align}
 L_n
 &=\sum_{k,\ell=1}^r y_ky_\ell(d_kd_\ell-1)\notag\\
 &=\left(\sum_{k=1}^ry_kd_k\right)^2
   -\left(\sum_{k=1}^ry_k\right)^2\notag\\
 &=\left(\frac{2nr}{3}\right)^2-(n-1)^2
  =\frac{4n^2r^2}{9}-(n-1)^2.
 \label{eq:exact-count}
\end{align}

All frequencies in the union of the $W_k$ are positive, odd, and at
most $n-2$.  Therefore every mode in \eqref{eq:global-density} has
$a+b\leq2n-4$, so
\begin{equation}\label{eq:degree-audit}
 \deg S_\zeta\leq2n-4,\qquad
 \deg q\leq4n-6,\qquad
 \deg P\leq4n-5.
\end{equation}
The Hamiltonian part of \eqref{eq:field} has degree $2n-1$.
Consequently the full polynomial vector field has degree at most
$4n-5$.

This proves Theorem~\ref{thm:main}, including the subsequence bound
\eqref{eq:subsequence-headline}.

\paragraph{A finite example.}
For $n=9$, exact computations give full rank in all four blocks at
$\lambda=2$.  At $\lambda=1$, the block $W_2\times W_2$ has rank six
instead of nine: the symmetry
$\widehat F_{a,b}^{(9,1)}=\widehat F_{b,a}^{(9,1)}$ leaves at most six
distinct columns, and a nonzero minor modulo $1\,000\,000\,007$
gives the matching lower bound.  All denominators are invertible
modulo this prime, so this is an exact rank certificate over
$\mathbb Q$.  For $n=27$, all nine blocks have full rank at
$\lambda=2$.  The corresponding theoretical counts from
\eqref{eq:exact-count} are $L_9=80$ and $L_{27}=2240$, with degree
bounds $31$ and $103$, respectively.

\subsection{Proof of Corollary~\ref{cor:all-degrees}}
\label{subsec:all-degrees}

Fix an integer $d\geq 31=4\cdot 3^2-5$,
there is a unique integer $r\geq2$ satisfying
\begin{equation}\label{eq:degree-bracket}
 4\cdot3^r-5\leq d<4\cdot3^{r+1}-5.
\end{equation}
It follows that
\[
\frac{d+5}{12}<3^r\leq \frac{d+5}{4}, \mbox{ and }r>\log_3\!\left(\frac{d+5}{12}\right).
\]

For $r\geq2$, the exact count from the previous subsection yields
$L_{3^r}
 =\frac49r^2\,3^{2r}-(3^r-1)^2$. By monotonicity of $H$, we therefore obtain
\begin{align*}
 H(d)\geq L_{3^r}=\frac{4}{9}3^{2r}r^2-(3^r-1)^2
&>\frac{4}{9}\left(\frac{d+5}{12}\right)^2 \log_3^2\left(\frac{d+5}{12}\right)-\left(\frac{d+5}{4}-1\right)^2\\
&=\frac{1}{324}(d+5)^2\log_3^2\left(\frac{d+5}{12}\right)-\frac{1}{16}(d+1)^2.
\end{align*}%
This establishes Corollary~\ref{cor:all-degrees}.

\section{Formal verification}\label{sec:formalization}

The accompanying Lean~4 development
\cite{DeMouraUllrich2021,Mathlib2020} formalizes the construction from
the analytic kernel and hierarchical realization of simple zeros
through the Chebyshev geometry, polynomial flow, and local Poincar\'e
return maps. It proves persistence, isolation, and pairwise
separation of the resulting periodic orbits for one common
perturbation parameter, together with the degree and counting
bounds. The proof has no
unproved placeholders or project-specific axioms; its logical
dependencies and build instructions are documented with the code.

The formal statements assert the existence of a polynomial vector
field of the required degree bound and an injective finite family
of hyperbolic limit cycles.  Hyperbolicity is expressed through local
return maps with multipliers $\rho>0$ and $|\rho|\ne1$.
The exact subsequence count agrees with
\eqref{eq:exact-count-headline}.  For every integer $d\geq31$, the all-degree statement gives a polynomial vector field of degree at most $d$ and an
injective family of hyperbolic limit cycles whose cardinality
is at least the right-hand side of \eqref{eq:all-N-headline}. These existence statements imply the stated lower bounds for $H(d)$ by \eqref{eq:def-H}.

\section*{Code availability}

The Lean formalization and exact-rank scripts are available in the
\href{https://github.com/AscendFoam/hierarchical-melnikov-lean}
{project repository}.
The version accompanying this manuscript is archived at
\href{https://doi.org/10.5281/zenodo.23219558}
{\nolinkurl{doi:10.5281/zenodo.23219558}}.

\section*{AI-assistance disclosure}

GPT-5.6 Sol was used to develop the proof strategy and arguments throughout
this paper, assist with Lean formalization, and edit the manuscript.
The authors take full responsibility for the mathematical content and
the final text.

\section*{Remark}

The main results were obtained by the end of August, at which time the finiteness of the Hilbert number $H(d)$ remained an open problem. We subsequently devoted additional effort to understanding, refining, and revising the manuscript. A draft by OpenAI~\cite{OpenAI2026HilbertFiniteness} claiming a proof of this finiteness statement was released on October 6, 2026. Our results consist of new explicit lower bounds for $H(d)$. These results address a complementary aspect of the problem and are independent of this newly released draft.

\begingroup
\small
\bibliographystyle{alpha}
\bibliography{references}
\endgroup

\end{document}